\documentclass[12pt]{amsart}

\usepackage{wrapfig}
\usepackage{amsmath, amssymb}
\usepackage{amscd}
\usepackage{enumitem}
\usepackage[dvipdfmx]{graphicx}
\usepackage[dvipdfmx]{color}
\usepackage[all]{xy}
\theoremstyle{definition}
\newtheorem{thm}{Theorem}[section]

\newtheorem{lem}[thm]{Lemma}
\newtheorem{cor}[thm]{Corollary}
\newtheorem{prop}[thm]{Proposition}
\newtheorem{qest}[thm]{Question}
\newtheorem{conj}[thm]{Conjecture}
\usepackage{xpatch}
\usepackage{xcolor}
\usepackage{tikz}
\xpatchcmd{\proof}
  {\itshape}
  {\bfseries}
  {}
  {}
\theoremstyle{definition}

\newtheorem{dfn}[thm]{Definition}
\newtheorem{example}[thm]{Example}

\numberwithin{equation}{section}

\usepackage{mathtools}
\title[Analysis of Contraction Mappings]{Analysis of Contraction Mappings\\to The Complement of Closed Curves}

\author{Shunichiro Orikasa}
\address{Department of Mathematics, Graduate School of Science, Kyoto University, Sakyoku, Kyoto 606–8502, Japan}
\email{orikasa.shunichiro.34x@st.kyoto-u.ac.jp}
\date{\today}
\begin{document}
\maketitle
\begin{abstract}
We study some analytic properties of distance decreasing self-maps onto the complement of a smooth curve $\Sigma$ in $S^n$. For $n>4$ and $n\equiv 0 \mod 4$, let $\Sigma$ be an embedded circle in $S^n$ and let $g$ be a complete Riemannian metric on $X=S^n\backslash \Sigma$ and $f:(X,g)\to (X,g_{std})$ be a 1-contracting diffeomorphism. We verify the sharp estimate $\inf_{x\in X}Sc(g)_x<n(n-1)$ if  any real Lipschitz 2-chain $C$ which represents the unit element $[C]$ in $H_2(S^n, W(\Sigma); \mathbb{R})$ satisfies 
$Area_g(C)>C(n)\cdot \max_i\{|\theta_i|\}$
where $W(\Sigma)$ is any tubular neighborhood of $\Sigma$ and $\{e^{2\pi i\theta_i}\}_i$ are the holonomy parameters along $\iota^*S^+$ where $S^+$ is the positive spinor bundle over $S^n$. This answers a question in \cite{gromov2018metric}.
\end{abstract}
\tableofcontents
\section{Introduction}
Llarull's theorem states the following rigidity result.  
\begin{thm}(\cite{llarull1998sharp})
  Let $M$ be a compact Riemannian spin manifold of dimension $n$.
Suppose there exists a $1$-contracting map 
\[f:(M,g)\to\mathbb{S}^n\]
 of non-zero degree and 
\[Sc(g)_x \geq n(n-1)\] 
for all $x\in M$.
Then $f:(M,g)\xrightarrow{\cong}\mathbb{S}^n$ is an isometry.
\end{thm}
Recently, Gromov asked the following question in \cite{gromov2018metric}, p.39. 
\begin{qest}
Let $Z\subset S^n$ be a closed subset of codimension $k\geq 2$, let $X$ be an orientable $n$-dimensional Riemannian manifold and let $f :X\to S^n\backslash Z$ be a smooth proper map of non-zero degree which is distance decreasing or,
more generally, area decreasing.
When and how can one bound the scalar curvature of $X$ ?
\end{qest}
Gromov claimed the following theorem in \cite{gromov2018metric} with a sketch of proof. We give a complete proof in this paper.
\begin{thm}
Suppose $n=2m$ and let $\Sigma\subset S^n$ be a 1 dimensional compact  connected manifold (i.e. $\Sigma$ is diffeomorphic to [0,1] or $S^1$), such that the monodromy transformation of $P_{Spin(n)}(S^n)|_{\Sigma}$ is trivial.\\
Let $(X,g)$ be a spin complete Riemannian manifold of dimension $n$.
\[\inf_{x\in X}Sc(g)_x < n(n-1)\]
holds if $f: X\to S^n\backslash\Sigma$ be a smooth proper $1$-contracting non-zero degree map.
\end{thm}
In \cite{gromov2018metric}, p.39, Gromov also asked whether the inequality $\inf Sc(X) < n(n-1)$ holds for smooth
closed curves $Z\subset S^n$, $n\geq3$ with non-trivial holonomy.

We prove our main theorem which gives an answer  to the above question. 
\begin{thm}\label{main2}
There exists a universal constant $C(n)$ such that the following statement holds.
Let $\Sigma$ be an embedded circle $\iota: S^1\hookrightarrow S^n$ and $X:=S^n\backslash\Sigma$.
Let $g$ be a complete Riemannian metric on $X$ and let $S^+$ be the positive spinor bundle over $S^n$. Then the estimate
\[\inf_{x\in X}Sc(g)_x<n(n-1)\]
holds if the following three conditions hold.
\begin{enumerate}[label=(\arabic*)]
\item $n>4$ and $n\equiv 0 \mod 4$.
\item Any real Lipschitz chain $C$ which represents the unit element $[C]$ in $H_2(S^n, W(\Sigma); \mathbb{R})$ satisfies 
\[Area_g(C)>C(n)\cdot \max_i\{|\theta_i|\},\]
where $W(\Sigma)$ is any tubular neighborhood of $\Sigma$ and $\{e^{2\pi i\theta_i}\}_i$ are the holonomy parameters along $\iota^*S^+$.
\item $f:(X,g)\to (X,g_{std})$ is a 1-contracting diffeomorphism.
\end{enumerate}
\end{thm}
The universal constant $C(n)$ is approximately $C(n)\approx 2^n$.
The proof is given in Section 4.4.

\subsection{Background Results on Fredholm Theory on Complete Riemannian Manifolds}
Recall that an oriented manifold $M$ admits a spin structure if and only if its second Stiefel- Whitney class is zero. Furthermore, if $w_{2}(M) = 0$,
then the spin structures on $M$ are in one-to-one correspondence with elements
of $H^1(M;\mathbb{Z}/2)$. In particular the $n$-sphere $S^n$ admits the unique spin structure if $n\ge 2$. We denote by $P_{spin(n)}(S^n)$ the unique spin structure.\\\indent
If $M$ is a spin manifold and  a Riemannian metric $g$ is chosen on $M$, the canonical connection on $P_{Spin(n)}(M)$ is 
induced by the Levi-Civita connection on $M$. It also induces the connection on the spinor bundle $S$. Therefore $S$ becomes a hermitian vector bundle equipped with a connection.\\\indent
If $M$ is a spin manifold of dimension $2n$, the volume element $\omega=i^n\mathrm{vol}\in \mathbb{C}l(TM)$ gives a $\mathbb{Z}/2$-grading on $S$ such that
\[S=S^{+}\oplus S^{-}\]
where $S^{+}$ and $S^{-}$ are the $+1$ and $-1$ eigenbundles for the Clifford multiplication by $\omega$.
The Dirac operator is odd with respect to this grading\[D^{\pm}:\Gamma(S^{\pm})\to \Gamma(S^{\mp})\]
where $D^{+}$ and $D^{-}$ are formal adjoints of one another.
\\\indent
We now recall the Lichnerowicz formula for the twisted Dirac operator $D_{E}$ of the bundle $S\otimes E$ over $M$. Let $(X,g)$ be a Riemannian spin manifold with scalar curvature $Sc(g)$, and let $S\otimes E$ be a twisted spinor bundle over X. Then the twisted Dirac operator $D_{E}$ and the connection Laplacian $\nabla^{*}\nabla$ of $S\otimes E$ satisfy the identity:
\begin{equation}\label{lich}
(D_{E})^2=\nabla^{*}\nabla+\frac{1}{4}Sc(g)+\mathcal{R}^E
\end{equation}
where $\mathcal{R}^E$ is defined in terms of a local basis of pointwise orthonormal tangent vector fields by $\mathcal{R}^E (\sigma\otimes\epsilon)=\frac{1}{2}\sum_{j,k=1}^n (e_j e_k\sigma)\otimes(R_{e_j,e_k}^{E}\epsilon)$ on vectors $\sigma\otimes\epsilon \in S\otimes E$.\\
Let $D_E$ be a twisted Dirac operator on a complete Riemannian spin manifold $X$. We say that $D_E$ is strictly positive at infinity if there exists a positive real number $c>0$ and a compact set $K\subset X$ such that
\[\frac{1}{4}Sc(g)+\mathcal{R}^E \geq c\ \text{(on $X\backslash K$).}\]
If $D_E$ is strictly positive at infinity, then it is Fredholm, as follows.
\begin{prop}[\cite{gromov1983positive} Theorem 3.2]\label{Fredholm}
Let $D_E$ be a twisted Dirac operator on a complete Riemannian spin manifold $X$ of dimension $2m$. If $D_E$ is strictly positive at infinity, then \[\mathrm{dim}\{s\in L^2(S^{+}\otimes E); D_Es=0\}\] and \[\mathrm{dim}\{s\in L^2(S^{-}\otimes E); D_Es=0\}\] are finite.
\end{prop}
We define $\mathrm{ind}(D_E^{+}):=\mathrm{dim}\{s\in L^2(S^{+}\otimes E); D_Es=0\}-\mathrm{dim}\{s\in L^2(S^{-}\otimes E); D_Es=0\}$.\\
Gromov and Lawson showed in \cite{gromov1983positive} the following theorem called the relative index theorem.
\begin{thm}[Gromov-Lawson]
Let $D_E$ and $D_F$ be twisted Dirac operators on a complete spin Riemannian manifold $X^{2m}$, and suppose that $E$ and $F$ coincide outside a compact set and that these operators are strictly positive at infinity. Then
\[\mathrm{ind}(D_E^{+})-\mathrm{ind}(D_F^{+})=\int_X \hat{A}(TX)\wedge (ch(E)-ch(F))\]
holds.
\end{thm}

\subsection{Definitions and Notations}
We fix some notations and introduce some definitions.
\begin{itemize}
\item Let $(X,g)$ be a Riemannian manifold. We denote by $Sc(g)$ the scalar curvature function of $X$.
\item $\mathbb{S}^n:=(S^n, g_{std})$ is denoted by the unit sphere with the standard metric.
\item Let $E$ be the complex vector bundle. We define the reduced Chern character $\hat{ch}(E)$ as
\[\hat{ch}(E):=ch(E)-rank(E)\]
where $ch(E)$ is the Chern character of $E$.
\item We often use the associated bundle 
\begin{equation}\label{E0}
E_{0}:=P_{spin(n)}(S^n)\times_{\lambda}\mathbb{C}l_n
\end{equation} where $P_{spin(n)}(S^n)$ is the unique spin structure of the $n$-sphere $S^n\ (n\geq2)$ and $\lambda$ is the representation by left multiplication.
\end{itemize}
\begin{dfn}[dilation constant]
Let $f:(X,g_X)\to (Y,g_Y)$ be a differentiable map between smooth Riemannian manifolds.
The dilation constant of $f$ is defined by the supremum of the norms of the differential $Df$.
\[dil(f):=sup_{p\in X}||Df_p||\in\mathbb{R}\cup\{\infty\}\]

We say $f$ is \textbf{strictly contracting at infinity} if there exists a compact subset $K\subset X$ such that 
\[dil(f|_{X\backslash K})<1\]
hold.
\end{dfn}
We constructed the collapsing map $\epsilon_\delta$ by the similar construction of \cite{chu2024llarullstheorempuncturedsphere}.
\begin{dfn}\label{epsilondelta}Let $\Sigma$ be a 1 dimensional connected compact manifold. For a sufficiently small $\delta>0$, we define $\epsilon_{\delta}:\mathbb{S}^n\to \mathbb{S}^n$ as follows. 
We say that the map $\epsilon_\delta$ is the \textbf{collapsing map} of $\Sigma$ if it satisfies the following properties. 
\begin{enumerate}
\item There exist open neighborhoods \[W_\delta(\Sigma)\subsetneq W'_\delta(\Sigma)\] 
of $\Sigma$ such that \[\epsilon_\delta(W_\delta(\Sigma))=\Sigma\]
\item $\epsilon_{\delta}(x)=x$ outside of $W'_\delta(\Sigma)$ 
\item $\displaystyle\limsup_{\delta \to 0} dil(\epsilon_\delta)\leq 1$
\end{enumerate}
\end{dfn}
\begin{prop}
There exists a collapsing map of $\Sigma$.
\end{prop}
\begin{proof}
Let $d:S^n\times S^n\to\mathbb{R}$ be the metric induced by the Riemannian metric $g_{std}$.
\begin{enumerate}[label=$\left(\arabic*\right)$]
\item First, assume that $\Sigma\cong S^1$. Set \[N\Sigma:=\{v\in TS^n; v\perp T_x\Sigma \ \text{for some}\  x\in\Sigma\}\]
\[N_{2\delta}:=\{v\in N\Sigma; ||v||<2\delta \}\]
\[U_{2\delta}:=\{p\in S^n; d(p,\Sigma)<2\delta \}\]
By the tubular neighborhood theorem, fix a sufficiently small $2\delta$, the map $exp$ is a diffeomorphism from $N_{2\delta}$ onto $U_{2\delta}$.
There exists a function $\eta:[0,2\delta]\to \mathbb{R}$ such that
\begin{itemize}
\item $\eta(t)=0$\ for $t\in[0,\frac{\delta^2}{2(1+\delta)}]$
\item $\eta(t)=t$\ for $t\in[\delta,2\delta]$
\item $0\leq \eta'(t)\leq 1+\delta$\ for $t\in[0,2\delta]$
\end{itemize}
By choosing a orthonormal framing of $N_{2\delta}$, we have an isometry
\[f:N_{2\delta}\xrightarrow{\cong}\Sigma\times B^{n-1}(2\delta)\]
where $B^{n-1}(2\delta)$ is a $n-1$ dimensional Euclidean ball with radius $2\delta$.\\
Set $g:\Sigma\times B^{n-1}(2\delta)\to\Sigma\times B^{n-1}(2\delta),(p,r\theta)\mapsto(p,\eta(r)\theta)$ where $r\theta$ is expression by the radial coordinates.

Define $\epsilon_\delta$ through 
\[U_{2\delta}\xrightarrow{exp^{-1}}N_{2\delta}\xrightarrow{f}\Sigma\times B^{n-1}(2\delta)\xrightarrow{g}\Sigma\times B^{n-1}(2\delta)\xrightarrow{f^{-1}}N_{2\delta}\xrightarrow{exp}U_{2\delta}\]
We have \[dil(g)\leq 1+\delta\] and \[dil(exp:N_{2\delta}\to U_{2\delta})\to 1\ (\delta\to 0)\]\[dil(exp^{-1}:U_{2\delta}\to N_{2\delta})\to 1\ (\delta\to 0)\]
Thus $\displaystyle\limsup_{\delta \to 0} dil(\epsilon_\delta)\leq 1$.\\
$\epsilon_\delta$ is an identity map near the boundary of $U_\delta$. Therefore $\epsilon_{\delta}$ smoothly extends to a map on $\mathbb{S}^n$ by defining $\epsilon_\delta(x):=x\ (\text{for}\ x\in\mathbb{S}^n\backslash U_\delta)$. Then we obtain the map $\epsilon_\delta:\mathbb{S}^n\to\mathbb{S}^n$ such that 
$\displaystyle\limsup_{\delta \to 0} dil(\epsilon_\delta)\leq 1$.

\item Secondly, suppose that $\Sigma\cong[0,1]$. There exists an embedded circle $S^1\cong\tilde{\Sigma}\subset\mathbb{S}^n$ such that $\Sigma\subset\tilde{\Sigma}$ (by concatenating two points $p_{\pm}$ of $\partial\Sigma$). Repeat the same argument for $\tilde{\Sigma}$ but replace $g$. Set
\[N\tilde{\Sigma}:=\{v\in TS^n; v\perp T_x\tilde{\Sigma} \ \text{for some}\  x\in\tilde{\Sigma}\}\]
\[\tilde{N}_{2\delta}:=\{v\in N\tilde{\Sigma}; ||v||<2\delta \}\]
\[\tilde{U}_{2\delta}:=\{p\in S^n; d(p,\tilde{\Sigma})<2\delta \}\]
By the tubular neighborhood theorem, fix a sufficiently small $2\delta$, the map $exp$ is a diffeomorphism from $\tilde{N}_{2\delta}$ onto $\tilde{U}_{2\delta}$.
There exists a function $\eta:[0,2\delta]\to \mathbb{R}$ such that
\begin{itemize}
\item $\eta(t)=0$\ for $t\in[0,\frac{\delta^2}{2(1+\delta)}]$
\item $\eta(t)=t$\ for $t\in[\delta,2\delta]$
\item $0\leq \eta'(t)\leq 1+\delta$\ for $t\in[0,2\delta]$
\end{itemize}
By choosing an orthonormal frame of $\tilde{N}_{2\delta}$, we have an isometry
\[f:\tilde{N}_{2\delta}\xrightarrow{\cong}\tilde{\Sigma}\times B^{n-1}(2\delta)\]
Construct $\tilde{g}:\tilde{\Sigma}\times B^{n-1}(2\delta)\to\tilde{\Sigma}\times B^{n-1}(2\delta)$ as follows.\\
For $(p,r\theta)\in \Sigma\times B^{n-1}(2\delta)$, we define
\[\tilde{g}(p,r\theta)=(p,\eta(r)\theta)\]
Let $B^n_{\pm}\subset\Sigma\times B^{n-1}(2\delta)$ be half balls around $p_{\pm}$. (See below.)\\
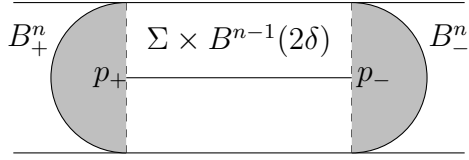
\begin{center}
\begin{tikzpicture}
\draw(-3,2) -- (3, 2);
\draw(-3,0) -- (3, 0);
\draw(-1.5,1) -- (1.5, 1);
\draw[dashed](-1.5,2)--(-1.5,0);
\draw[dashed](1.5,2)--(1.5,0);
\draw[fill=lightgray, xshift=1.5cm, yshift=1cm](90:1) arc (90:-90:1cm);
\draw[fill=lightgray, xshift=-1.5cm, yshift=1cm](90:1) arc (90:270:1cm);
\draw(-2.8cm,1.5) node{$B_{+}^n$};
\draw(2.8cm,1.5) node{$B_{-}^n$};
\draw(-1.7cm,1) node{$p_{+}$};
\draw(1.8cm,1) node{$p_{-}$};
\draw(0cm,1.5) node{$\Sigma\times B^{n-1}(2\delta)$};
\end{tikzpicture}
\end{center}
For $r\theta\in B^n_{\pm}$ (expressed by radial coordinates), we define
\[\tilde{g}(r\theta)=\eta(r)\theta\]
For $x\in\tilde{\Sigma}\times B^{n-1}(2\delta)\backslash(B_{+}\cup B_{-})$, define $\tilde{g}(x)=x$.\\
Define $\epsilon_\delta$ through 
\[\tilde{U}_{2\delta}\xrightarrow{exp^{-1}}\tilde{N}_{2\delta}\xrightarrow{f}\tilde{\Sigma}\times B^{n-1}(2\delta)\xrightarrow{g}\tilde{\Sigma}\times B^{n-1}(2\delta)\xrightarrow{f^{-1}}\tilde{N}_{2\delta}\xrightarrow{exp}\tilde{U}_{2\delta}\]
We have \[dil(g)\leq 1+\delta\] and \[dil(exp:\tilde{N}_{2\delta}\to \tilde{U}_{2\delta})\to 1\ (\delta\to 0)\]\[dil(exp^{-1}:\tilde{U}_{2\delta}\to \tilde{N}_{2\delta})\to 1\ (\delta\to 0)\]
Thus $\displaystyle\limsup_{\delta \to 0} dil(\epsilon_\delta)\leq 1$.\\
$\epsilon_\delta$ is an identity map near the boundary of $\tilde{U}_\delta$. Therefore $\epsilon_{\delta}$ extends smoothly to a map on $\mathbb{S}^n$ by defining $\epsilon_\delta(x):=x\ (\text{for}\ x\in\mathbb{S}^n\backslash U_\delta)$. Then we obtain the map $\epsilon_\delta:\mathbb{S}^n\to\mathbb{S}^n$ such that 
$\displaystyle\limsup_{\delta \to 0} dil(\epsilon_\delta)\leq 1$.
\end{enumerate}

\end{proof}

We define $\delta(\Sigma)$ as the supremum of $\delta$ such that the collapsing map $\epsilon_{\delta}$ is defined. Precisely, if $\Sigma\cong S^1$(resp. $[0,1]$)
\begin{equation}\label{delta}
\delta(\Sigma):=sup\{\delta>0; exp: N_{2\delta}\to U_{2\delta} \ \text{is a diffeomorphism}\}
\end{equation}
(resp. $\delta(\Sigma):=sup\{\delta>0; exp: \tilde{N}_{2\delta}\to \tilde{U}_{2\delta} \ \text{is a diffeomorphism}\}$)
\subsection{Llarull's theorem}
Llarull showed in \cite{llarull1998sharp} the famous rigidity theorem.
\begin{thm}(\cite{llarull1998sharp})
  Let $M$ be a compact Riemannian spin manifold of dimension $n$.
Suppose there exists a $1$-contracting map 
\[f:(M,g)\to\mathbb{S}^n\]
 of non-zero degree.\\
Then, either there exists $x\in M$ with $Sc(g)_x<n(n-1)$, or $f:M\to S^n$ is an Riemannian isometry.

\end{thm}
We use the following estimates that appeared in the proof of Theorem \ref{Llarull}.
\begin{prop}(\cite{llarull1998sharp})\label{Llarull}
Suppose $n=2m$. Let $M$ be a spin $n$-dimensional Riemannian
  manifold with metric $g$. Let $\mathbb{S}^{n}$ be the unit $n$-sphere with standard metric $g_{std}$.
  Let $f:M\to S^{n}$ be a smooth $1$-contracting map and set $E_0:=P_{spin(n)}(S^n)\times_{\lambda}\mathbb{C}l_n$, $E:=f^*E_0$. 
  Then 
  \[\int_{S^{n}}ch_m(E_0^+)\neq 0\]
  and
  \[\mathcal{R}_x^{E}\geq-\frac{1}{4}n(n-1)\]
  for all $x\in M$.
\end{prop}

\section{Smooth Curves with Trivial Monodromies}
In this section, we give a complete proof of the theorem which is claimed by Gromov in \cite{gromov2018metric}.

\begin{thm}\label{tri}
Suppose $n=2m$ and let $\Sigma\subset S^n$ be a 1 dimensional compact  connected manifold(i.e. $\Sigma$ is diffeomorphic to [0,1] or $S^1$), such that the monodromy transformation of $P_{Spin(n)}(S^n)|_{\Sigma}$ is trivial. \\
Let $(X,g)$ be a spin complete Riemannian manifold of dimension $n$.
\[\inf_{x\in X}Sc(g)_x < n(n-1)\]
holds if $f: X\to \mathbb{S}^n\backslash\Sigma$ be a smooth proper $1$-contracting non-zero degree map.\\

\end{thm}
\begin{proof}
We can assume $f$ is strictly contracting at infinity because we can replace $f$ by $f'=h\circ f$ where $h$ is a self-map on $\mathbb{S}^n$ which is contracting except for a neighborhood of some point. (Notice that there exists a point $p$ such that $dil(f)_p<1$.) In this step, $\Sigma$ is changed into another one $\Sigma'$ and the holonomy parameter becomes slightly different from the former one. Then we can perturb near some point on $\Sigma'$ so that the holonomy along $\Sigma'$ becomes trivial and $f$ is strictly contracting.\\
By contradiction, assume that $Sc(g)\geq n(n-1)$. Let $\epsilon_{\delta}:\mathbb{S}^n\to\mathbb{S}^n$ is the collapsing map constructed in Definition \ref{epsilondelta} and let $h:X\to\mathbb{S}^n$ be the composition $h:=\epsilon_{\delta}\circ f$. It can be checked that $dil(h)$ satisfies $dil(h)\leq 1$ for sufficiently small $\delta$ because we have the inequalities $dil(h)\leq dil(\epsilon_{\delta})\cdot dil(f)$ and $f$ is strictly contracting at infinity.

We take $E_{0}:=P_{spin(n)}(S^n)\times_{\lambda}\mathbb{C}l_n$ (see \eqref{E0})and obtain a pull-back bundle $E:=h^{*}(E_0)$ with connection. 
Since $E_0$ has the $\mathbb{Z}/2$-grading $E_0=E_0^+\oplus E_0^-$, $E$ splits into $E=E^+\oplus E^-$ where 
$E^{\pm}=h^{*}(E_0^{\pm})$. According to Proposition \ref{Llarull}, we have 
\begin{equation}
\int_{S^{2m}}ch_m(E_0^+)\neq 0
\end{equation} and 
\begin{equation}
\mathcal{R}^{E}\geq -\frac{1}{4}n(n-1)
\end{equation}
The connection on $E_0^{+}|_{\Sigma}$ is trivial because of the monodromy condition. Thus $E_0^{+}|_{\Sigma}$ is isomorphic to the trivial vector bundle with connection.
\[(E_0^{+}|_{\Sigma}, \nabla)\cong (\underline{\mathbb{C}}^{N}, d)\]
We also have 
\[(E^{+}|_{f^{-1}(W_\delta(\Sigma))}, \nabla)\cong (\underline{\mathbb{C}}^{N}, d)\]
because $h$ maps $f^{-1}(W_\delta(\Sigma))$ onto $\Sigma$ and the connection on $E|_{f^{-1}(W_\delta(\Sigma))}$ is the 
pull-back connection. Note that $f^{-1}(S^{n}\backslash W_\delta(\Sigma))$ is a compact set in $X$ since $f$ is proper.\\
\indent In order to apply the relative index theorem, we check the Fredholmness of the operators  $D_{\underline{\mathbb{C}}^{N}}$ and $D_E$. By the Lichnerowicz theorem,

\begin{equation}\label{lich c}
(D_{\underline{\mathbb{C}}^{N}})^2=\nabla^*\nabla+\frac{1}{4}Sc(g)
\end{equation}
\begin{equation}\label{lich e}
(D_E)^2=\nabla^*\nabla+\frac{1}{4}Sc(g)+\mathcal{R}^E
\end{equation}
Note that the potential terms of both operators are strictly positive at the end of $X$ because 
\[Sc(g)\geq n(n-1)\]
and \[R^E\equiv 0\]
since the vector bundle $E$ is flat on $f^{-1}(W_\delta(\Sigma))$. Thus $D_{\underline{\mathbb{C}}^{N}}$ and $D_E$ are 
Fredholm operators (by Proposition\ref{Fredholm}). By the relative index theorem,
\begin{equation}
\mathrm{ind}(D_E^{+})-\mathrm{ind}(D_{\underline{\mathbb{C}}^{N}}^{+})=\int_X \hat{A}(TX)\wedge (\hat{ch}(E))
\end{equation}
Recall that $ch(E)=rank(E)+ch_1(E)+\dots ch_m(E)$ with $ch_i(E) \in H^{2i}(X)$. On $\mathbb{S}^{2m}$ the Chern character of
 the vector bundle $E_0^+$ is given by,
 \[ch(E_0^+)=2^{n-1}+ch_m(E_0^+)\]
 hence
 \[ch(E)=h^*ch(E_0^+)=2^{n-1}+h^*ch_m(E_0^+)\]
 Consequently, using the fact that the leading term of $\hat{A}(TX)$ is 1 and $deg(f)\neq 0$,
 \begin{align}\label{even rel}
 \mathrm{ind}(D_E^{+})-\mathrm{ind}(D_{\underline{\mathbb{C}}^{N}}^{+})&=\int_X \hat{A}(TX)\wedge \hat{ch}(E)\\
&=\int_X \hat{A}(TX)\wedge h^*ch_m(E_0^+)\notag \\
&=\int_X \hat{A}(TX)\wedge f^*(\epsilon_\delta^*ch_m(E_0^+))\notag \\
&=deg(f)\int_{S^{2m}} ch_m(E_0^+)\neq 0\notag
\end{align}
On the other hand, comparing \eqref{lich c}, \eqref{lich e} and Llarull's estimate,
\[\frac{1}{4}Sc(g)\geq \frac{1}{4}n(n-1)\]
\[\frac{1}{4}Sc(g)+\mathcal{R}^E\geq 0\]
and \[\frac{1}{4}Sc(g)+\mathcal{R}^E\geq \frac{n(n-1)}{4}\] on $f^{-1}(W_\delta(\Sigma))$. Thus $\mathrm{ind}(D_{\underline{\mathbb{C}}^{N}}^{+})=\mathrm{ind}(D_E^{+})=0$ which contradicts \eqref{even rel}. 
\end{proof}
In particular, we obtain the following corollary.
\begin{cor}
Suppose $n=2m$ and let $\Sigma\subset S^n$ be a 1 dimensional compact manifold(i.e. a connected component of $\Sigma$ is diffeomorphic to [0,1] or $S^1$), such that the monodromy transformation of $P_{Spin(n)}(S^n)|_{\Sigma}$ is trivial. \\
Let  $g$ be a complete Riemannian metric on $S^n\backslash\Sigma$.\\
If we have the following inequality,\[g_x\geq (g_{std})_x\] on $x\in S^n$.
Then $\inf_{x\in S^n\backslash\Sigma}Sc(g)_x< n(n-1)$.
\end{cor}
\begin{proof}
Set $X:=(S^n\backslash\Sigma,g)$ and $f:=\mathrm{Id}:(X,g)\to\mathbb{S}^n$. Use Theorem \ref{tri}.
\end{proof}
\section{Distance Decreasing Maps onto The Complement of A Circle}
In this section, we consider closed curves with non-trivial monodromy. 
\subsection{Good triplets}
We introduce the following technical notion.
\begin{dfn}
Let $\Sigma$ be an embedded circle $S^1\cong\Sigma\subset S^n$ and let $g$ be a complete Riemannian metric on 
$S^n\backslash\Sigma$. Let $f: (S^n\backslash\Sigma,g)\to (S^n\backslash\Sigma,g_{std})$ be a $1$-contracting diffeomorphism.\\\indent
We say that $(g, f, \Sigma)$ is \textbf{good} if and only if there exists a real number $0<\delta_0<\delta(\Sigma)$ ($\delta(\Sigma)$ is defined in \eqref{delta}) and a hermitian 
vector bundle $F$ over $S^n\backslash\Sigma$ with connection satisfying the following properties.
\begin{enumerate}
\item $F|_{f^{-1}(W_{\delta_0}(\Sigma))}\cong (\epsilon_{\delta_0}|_{W_{\delta_0}(\Sigma)\cap (S^n\backslash\Sigma)}\circ f)^*(E_0^+|_{\Sigma})$ (as hermitian vector bundle with connection)
\item  $||\mathcal{R}^{F}||\leq \frac{1}{4} n(n-1)$
\item $\displaystyle\int_{S^n\backslash\Sigma}\hat{A}(TX)\wedge\hat{ch}(F)\neq \int_{S^{n}} ch_{\frac{n}{2}}(E_0^+)$
\end{enumerate}
\end{dfn}

Later, we consider non-trivial examples in Theorem \ref{ext}. 
\begin{thm}\label{adm}
Let $\Sigma$ be an embedded circle $\Sigma\subset S^n$ and $g$ be a complete metric on $S^n\backslash\Sigma$. Let $f: (S^n\backslash\Sigma,g)\to (S^n\backslash\Sigma,g_{std})$ be a $1$-contracting diffeomorphism and strictly contracting at infinity. Suppose $n=2m$ and $(g, f, \Sigma)$ is good. Then the estimate
\[\inf_{x\in S^n\backslash\Sigma} Sc(g)_x<n(n-1)\]
holds.
\end{thm} 
\begin{proof}
By contradiction, suppose $\inf_{x\in S^n\backslash\Sigma} Sc(g)_x\geq n(n-1)$. Set $X:=(S^n\backslash\Sigma, g)$ and let $h:X\to S^n$ be the composition $h:=\epsilon_{\delta}\circ f$. Let $\epsilon_{\delta}:\mathbb{S}^n\to\mathbb{S}^n$ be the collapsing map.\\
By assumption, there exists a hermitian vector bundle $F$ with connection over $S^n\backslash\Sigma$ which satisfies the following properties.\\
\begin{enumerate}
\item $F|_{f^{-1}(W_{\delta_0}(\Sigma))}\cong (\epsilon_{\delta_0}|_{W_{\delta_0}(\Sigma)\cap (S^n\backslash\Sigma)}\circ f)^*(E_0^+|_{\Sigma})$
\item  $||\mathcal{R}^{F}||\leq \frac{1}{4} n(n-1)$
\item $\displaystyle\int_{S^n\backslash\Sigma}\hat{A}(TX)\wedge\hat{ch}(F)\neq \int_{S^{2m}} ch_m(E_0^+)$
\end{enumerate}
We choose $\delta(<\delta_0)$ so that $dil(h)$ satisfies $dil(h)\leq 1$. It is possible because \[dil(h)\leq dil(\epsilon_{\delta})\cdot dil(f)\]
We now set $E_0:=P_{Spin(n)}(S^n)\times_{\lambda}\mathbb{C}l_n$ and obtain the pull-back bundle $E=h^*E_0$ over $X$ with connection.\\
Recall that 
\[\displaystyle\int_{S^{n}} ch_m(E_0^+)\neq 0\]
\[\mathcal{R}_x^E\geq -\frac{1}{4}n(n-1) ,x\in X\]
Note that $E_0|_\Sigma$ is flat because $\mathrm{dim}\Sigma=1$ and the curvature 2-form of $E_0|_\Sigma$ vanishes. Therefore $E|_{f^{-1}(W_{\delta}(\Sigma))}$ is flat because $h$ maps $f^{-1}(W_{\delta}(\Sigma))$ onto $\Sigma$ and the connection on $E|_{f^{-1}(W_{\delta}(\Sigma))}$ is the pull-back connection.\\
To apply the relative index theorem, we first check that $D_{E^+}$ and $D_F$ are Fredholm.\\
By theorem \ref{lich},
\[(D_{E^+})^2=\nabla^*\nabla+\frac{1}{4}Sc(g)+\mathcal{R}^{E^+}\]
\[(D_F)^2=\nabla^*\nabla+\frac{1}{4}Sc(g)+\mathcal{R}^F\]
we have $\mathcal{R}^F\equiv 0$ and $\mathcal{R}^{E^+}\equiv 0$ on $W_\delta(\Sigma)\cap X$ because $F$ and $E^+$ are flat on $W_\delta(\Sigma)\cap X$. Thus $D_{E^+}$ and $D_F$ are Fredholm.\\
Secondly, we have $F|_{f^{-1}(W_{\delta}(\Sigma))}\cong (\epsilon_{\delta}|_{W_{\delta}(\Sigma)\cap (S^n\backslash\Sigma)}\circ f)^*(E_0^+|_{\Sigma})$ because we have (1) and $\epsilon_{\delta_0}|_{W_\delta(\Sigma)}=\epsilon_\delta|_{W_\delta(\Sigma)}$.\\
Therefore, by the relative index theorem,
\begin{equation}\label{rel 2}
\mathrm{ind}(D_{E^{+}}^{+})-\mathrm{ind}(D_F^{+})=\int_X \hat{A}(TX)\wedge (ch(E^{+})-ch(F))
\end{equation}
We first compute the right-hand side of \eqref{rel 2},
\begin{align}\label{even rel 2}
\int_X \hat{A}(TX)\wedge (ch(E^{+})-ch(F))&=\int_X \hat{A}(TX)\wedge (\hat{ch}(E^{+})-\hat{ch}(F))\\
&=\int_X \hat{A}(TX)\wedge f^*ch_m(E_0^+)-\int_X \hat{A}(TX)\wedge \hat{ch}(F)\notag \\
&=\int_{S^{2m}} ch_m(E_0^+)-\int_X \hat{A}(TX)\wedge \hat{ch}(F)\notag
\end{align}
On the other hand, $\mathrm{ind}(D_{E^{+}}^{+})=\mathrm{ind}(D_F^{+})=0$ because 
\[\frac{1}{4}Sc(g)+\mathcal{R}^{E^+}\geq 0\]
\[\frac{1}{4}Sc(g)+\mathcal{R}^F\geq 0\]
and
\[\frac{1}{4}Sc(g)+\mathcal{R}^{E^+}\geq \frac{1}{4}n(n-1) \ \text{on $W_{\delta}(\Sigma)\cap X$}\]
\[\frac{1}{4}Sc(g)+\mathcal{R}^{F}\geq\frac{1}{4}n(n-1) \ \text{on $W_{\delta}(\Sigma)\cap X$}\]
which is a contradiction. This completes the proof of Theorem \ref{adm}.
\end{proof}
\subsection{Computations of characteristic classes}
\begin{prop}\label{drham}
Let $\Sigma\subset S^n$ be an embedded circle and $n:=2m\geq 4$.
Suppose $X:=S^n\backslash\Sigma$ is diffeomorphic to $(\mathbb{R}^{n-1}\backslash\{0\})\times \mathbb{R}$ and $F$ is a Hermitian vector bundle over $X$ with connection such that $F|_{W_\delta(\Sigma)\cap X}$ is flat on $W_\delta(\Sigma)\cap X$ for a real number $\delta>0$.\\
Then
\[\int_X \hat{A}(TX)\wedge\hat{ch}(F)=\begin{cases}
 \int_X ch_m(F) & (n\equiv 0 \mod 4) \\
 \int_X ch_m(F)+\int_X \hat{A}_{\frac{n-2}{4}}(TX)\wedge ch_2(F) & (n\equiv 2 \mod 4)
\end{cases}\]
Note the above integral is well-defined because $\hat{ch}(F)$ is compactly supported.
\end{prop}
\begin{proof}
Using $X\cong (\mathbb{R}^{n-1}\backslash\{0\})\times \mathbb{R}\simeq S^{n-2}$, the de Rham cohomology group $H_{dR}^{*}(X)$ is $H_{dR}^i(X)=0$ when $1\leq i <n-2$.
Recall that $\hat{A}(TX)=1+\hat{A}_1(TX)+\dots +\hat{A}_{\lfloor\frac{m}{2}\rfloor}(TX)$.\\
If m is even $m=2k$,
There exists a differential form $\eta_{i}$ on $X$ such that 
\[\hat{A}_i(TX)=d\eta_i\]
for $1\leq i \leq k $.\\
\begin{align}
\int_X \hat{A}(TX)\wedge\hat{ch}(F)=&\int_X ch_m(F)+\sum_{1\leq i \leq k}\int_X \hat{A}_i(TX)\wedge \hat{ch}(F)\notag
\end{align}
and by the Stokes theorem, \[\int_X \hat{A}(TX)\wedge\hat{ch}(F)=\int_X d(\eta_i\wedge\hat{ch}(F))=0\] for $1\leq i \leq k$.
\\
If $m$ is odd $m=2k+1$,
There exists a differential form $\eta_{i}$ on $X$ such that 
\[\hat{A}_i(TX)=d\eta_i\]
for $1\leq i \leq k-1$.\\
\begin{equation}
\int_X \hat{A}(TX)\wedge\hat{ch}(F)=\int_X ch_m(F)+\sum_{1\leq i \leq \lfloor \frac{m}{2}\rfloor}\int_X \hat{A}_i(TX)\wedge \hat{ch}(F)+\int_X \hat{A}_k(TX)\wedge \hat{ch}(F)\notag
\end{equation}
Therefore 
\[\int_X \hat{A}(TX)\wedge\hat{ch}(F)=\int_X ch_m(F)+\int_X \hat{A}_k(TX)\wedge ch_2(F)\]
This completes the proof.
\end{proof}
Next we take the monodromy transformation of $P_{Spin(n)}(S^n)|_\Sigma$ into account. We first consider the flat connection on $E_0^+|_\Sigma$.

\begin{prop}\label{mon}
Let $(V, \nabla)$ be a Hermitian vector bundle over $S^1$ with connection ($\mathrm{rank}(V)=N$) and let $A\in U(N)$ be the image of the identity under the monodoromy homomorphism $\pi_1(S^1)\to U(N),1\mapsto A$.\\
Then there exist complex line bundles $\{(L^i,\nabla^i)\}$ with connection such that
\[V\cong \bigoplus_{i=1}^N (L^i,\nabla^i)\]
and\[(L^i, \nabla^i)\cong (\underline{\mathbb{C}},d-2\pi\sqrt{-1}\theta_i dt)\]
where $\{e^{2\pi\sqrt{-1}\theta_i}\}_{1\leq i\leq N}$ are eigenvalues of $A$. (We denote $Hol(V):=\{e^{2\pi\sqrt{-1}\theta_i}\}$.)
\end{prop}
\begin{proof}
Using the classification of flat bundles,
\[V\cong\mathbb{R}\times_{\mathbb{Z}}\mathbb{C}^N\]
as hermitian vector bundle over $\mathbb{R}/\mathbb{Z}\cong S^1$ with connection.\\
Here $\mathbb{Z}\curvearrowright\mathbb{C}^N$ is given by the multiplication of $A$, $n\cdot v=A^nv\ (n\in\mathbb{Z},\ v\in\mathbb{C}^N)$. The $\mathbb{C}^N$ is decomposed into eigenspaces of $A$.
\[\mathbb{C}^N=\bigoplus_{i=1}^N \mathbb{C}v_i\]
where eigenvalues and eigenvectors of $A$ are $\{e^{2\pi\sqrt{-1}\theta_i}\}_{1\leq i\leq N}$ and $Av_i=e^{2\pi\sqrt{-1}\theta_i}v_i\ (\theta_i\in\mathbb{R})$.\\
Set $L^i\cong\mathbb{R}\times_{\mathbb{Z}}\mathbb{C}_i$ and $\nabla^{i}$ denotes the induced connection by V.\\
Then
\[V\cong \bigoplus_{i=1}^N (L^i,\nabla^i)\]
We have a local parallel section $e_i$ of $L^i$ given by \[e_i([t])=[(t,1)]\ (t\in \mathbb{R})\]
On the other hand,
\[s_i([t]):=[(t,e^{-2\pi\sqrt{-1}\theta_it})]\ (t\in\mathbb{R})\]
gives a global trivialization of $L^i$ and
\begin{align}
\nabla_{\frac{d}{dt}}^{i}s_i&=\nabla_{\frac{d}{dt}}^{i}(e^{-2\pi\sqrt{-1}\theta_i t}e_i)\\
&=\frac{d}{dt}e^{-2\pi\sqrt{-1}\theta_i t}e_i+e^{-2\pi\sqrt{-1}\theta_it}\nabla_{\frac{d}{dt}}^{i}e_i\notag\\
&=-2\pi\sqrt{-1}\theta_is_i\notag
\end{align}
Therefore we have $(L^i,\nabla^{i})\cong(\underline{\mathbb{C}}^N,d-2\pi\sqrt{-1}\theta_i dt)$ and completes the proof.
\end{proof}

\subsection{Nice embeddings of $S^1$ into $S^n$}
\begin{dfn}
Let $\iota:S^1\to S^n$ be an embedding and set $\Sigma:=\iota(S^1)$, $X:=S^n\backslash\Sigma$. ($n\geq3$)\\
(1)We say $\iota$ is \textbf{nice} if it satisfies the following settings.\\
There exists a diffeomorphism $\phi:X\xrightarrow{\cong}S^{n-2}\times \mathbb{R}^2$ and $\delta'>0$ which satisfies 
\[\phi(W_\delta(\Sigma)\cap X)\subset S^{n-2}\times(\mathbb{R}^2\backslash\{0\})\]
for $\delta\in (0,\delta')$ and
$$
\begin{CD}
W_\delta(\Sigma)\cap X  @>\epsilon_{\delta}>> S^n \\
@V{\phi |_{W_\delta(\Sigma)\cap X}}VV @A{\iota}AA \\
 S^{n-2}\times(\mathbb{R}^2\backslash\{0\})@>\alpha>> S^1
\end{CD}
$$
where $\alpha:S^{n-2}\times(\mathbb{R}^2\backslash\{0\})\to S^1$ is the composition
\[S^{n-2}\times(\mathbb{R}^2\backslash\{0\})\to \mathbb{R}^2\backslash\{0\}\cong S^1\times \mathbb{R}\to S^1\]
(2)Suppose $\iota$ is nice and $f: (X, g)\to (X, g_{std})$ is a $1$-contracting diffeomorphism. We fix the notations $X^*:=\phi^{-1}(S^{n-2}\times(\mathbb{R}^2\backslash\{0\}))$, $X_0:=\phi^{-1}(S^{n-2}\times\{0\})$ and $\beta:=\alpha\circ\phi|_{X^*}\circ f$.
\end{dfn}
\begin{prop}\label{nice}
Suppose $n\geq 4$, every embedding $\iota:S^1\to S^n$ is nice.(\cite{hirsch2012differential})
\end{prop}
\begin{proof}
(\textbf{Step1.})
Let $\iota_0:S^1\to S^n\subset \mathbb{R}^{n+1}$ be the standard embedding \[\iota_0(t):=(\cos t, \sin t, 0)\]
Set \[\Sigma_0:=\iota_0(S^1)=\{x\in\mathbb{R}^{n+1}| (x_0)^2+(x_1)^2=1, x_2=\cdots=x_n=0\}\]
and\[S^{n-2}:=\{x\in\mathbb{R}^{n+1}| (x_2)^2+\cdots+(x_n)^2=1, x_0=x_1=0\}\]
First, we show that $\iota_0$ is nice. We define $h:S^{n-2}\times S^1\times [0,\frac{\pi}{2}]\to S^n$ by
\[h(p,q,t):=(\cos t)p+(\sin t)q\]
Here, we identified $S^1\cong\Sigma_0$. Let $\tilde{h}$ be the quotient map
\[\tilde{h}:S^{-2}\times S^1\times[0,\frac{\pi}{2}]/(p,q,0)\sim(p,q',0)\to X\]
Notice that $\tilde{h}$ is a homeomorphism. We define $\phi:X\to S^{n-2}\times\mathbb{R}^2$ through
\[X\xrightarrow{\tilde{h}^{-1}}S^{n-2}\times S^1\times [0,\frac{\pi}{2})/\sim\cong S^{n-2}\times\mathbb{R}^2\]
Notice that we can take $\phi$ as a diffeomorphism. There exists a $\delta'>0$ such that
\[\phi(W_\delta(\Sigma)\cap X)\subset S^{n-2}\times (\mathbb{R}^2\backslash\{0\})\]
for any $\delta\in (0,\delta')$ and 
$$
\begin{CD}
W_\delta(\Sigma)\cap X  @>\epsilon_{\delta}>> S^n \\
@V{\phi |_{W_\delta(\Sigma)\cap X}}VV @A{\iota}AA \\
 S^{n-2}\times(\mathbb{R}^2\backslash\{0\})@>\alpha>> S^1
\end{CD}
$$
commutes.\\
(\textbf{Step2.})
Let $\iota:S^1\to S^n$ be any embedding. We show that $\iota$ is nice. Recall the fact that $\iota$ is isotopic to $\iota_0$ when $n\geq 4$. Moreover, by using the uniqueness of tubular neighborhoods, we can assume that there exists a diffeomorphism $\tilde{F}:S^n\to S^n$ such that
\[\tilde{F}\circ\nu(\iota_0)=\nu(\iota)\circ F\]
where $\nu(\iota_0):N_{\delta'}\to W_\delta'(\Sigma_0)$ and $\nu(\iota):\tilde{N}_{\delta'}\to W_\delta'(\Sigma)$ are 
normal tubular neighborhoods. ($N_{\delta'}=\{v\in N\Sigma_0;||v||<\delta'\},\ \tilde{N}_{\delta'}=\{v\in N\Sigma;||v||<\delta'\}$) $F:N_{\delta'}\xrightarrow{\cong}\tilde{N}_{\delta'}$ is a fiber bundle map over $S^n$.\\
Notice that 
$$
\begin{CD}
W_\delta(\Sigma)  @>\epsilon_{\delta}>> S^n \\
@A{\nu(\iota)}AA @A{\iota}AA \\
\tilde{N}_{\delta}@>0>> S^1
\end{CD}
$$
and
$$
\begin{CD}
W_\delta(\Sigma_0)  @>\epsilon_{\delta}(\Sigma_0)>> S^n \\
@V{\tilde{F}}VV @V{F}VV \\
W_\delta(\Sigma)@>\epsilon_{\delta}(\Sigma)>> S^n
\end{CD}
$$
commute. We are ready to show that $\iota$ is nice. Set $\tilde{\phi}:=\phi\circ F$ then
$$
\begin{CD}
W_\delta(\Sigma)\cap X  @>\epsilon_{\delta}>> S^n \\
@V{\tilde{\phi }|_{W_\delta(\Sigma)\cap X}}VV @A{\iota}AA \\
 S^{n-2}\times(\mathbb{R}^2\backslash\{0\})@>\alpha>> S^1
\end{CD}
$$
for any $\delta\in(0,\delta')$.
\end{proof}
\begin{thm}\label{ext}
Let $n$ be a natural number such that $n>4, n\equiv0 \mod 4$. There exists a constant $C_n>0$ and the following statement holds. (The choice of $C_n$ is stated in the proof and only depends on the dimension $n$.)\\\indent
Let $\Sigma$ be an embedded circle $\iota: S^1\to S^n$ and $X:=S^n\backslash\Sigma$. Let $f: (X, g)\to (X, g_{std})$ be a $1$-contracting diffeomorphism and strictly contracting at infinity.\\\indent
Then $(g, f, \Sigma)$ is good if there exists a positive real number $\delta_0>0$ which satisfies the following conditions.
\begin{enumerate}[label=(\arabic*)]
\item $\beta^*(d\theta)|_{W_{\delta_0}(\Sigma)\cap X}\in\Omega^1(W_{\delta_0}(\Sigma)\cap X)$ extends smoothly to a differential form $\omega\in\Omega^1(X)$.
\item $\mathrm{max}\{|log(\xi_i)|; \xi_i\in Hol(E_0^+|_\Sigma)\}\cdot||d\omega||_\infty\leq C_n$
\end{enumerate}
\end{thm}
\begin{proof}
We decompose $\iota^*E_0^+$ into line bundles via the monodromy transformation,
\[\iota^*E_0^+\cong\bigoplus_i L^i\]
where $L^i=(\underline{\mathbb{C}}, d-2\pi\sqrt{-1}\theta_i d\theta)$ and $Hol(E_0^+|_\Sigma):=\{\xi_i\}=\{e^{2\pi\sqrt{-1}\theta_i}\}$. (Choose $\theta_i$ in order to attain the minimum $|log(\xi_i)|:=min\{|w|;\mathrm{exp}(w)=\xi_i\}$.) \\
We obtain the pull-back bundle over $f^{-1}(W_{\delta_0}(\Sigma))$
\[\tilde{L^i}:=\beta^*L^i|_{W_{\delta_0}(\Sigma)\cap X}=(\underline{\mathbb{C}}, d-2\pi\sqrt{-1}\theta_i\beta^*(d\theta))\]
Since $\tilde{L^i}$ is trivial, there is no obstruction to extend the vector bundle $\tilde{L^i}$ to a vector bundle over $X$. ($\tilde{L^i}$ denotes the vector bundle over $X$, abuse of notation.)\\
By assumption, we can extend the connection over X. (Set $\omega_i:=2\pi\sqrt{-1}\theta_i\omega$.)
\[\tilde{L^i}\cong (\underline{\mathbb{C}}, d-2\pi\sqrt{-1}\theta_i \omega)=(\underline{\mathbb{C}}, d-\omega_i)\]
Next we check the conditions of good. Set $F:=\bigoplus_i \tilde{L}^i$ and fix any sufficiently small $\delta_0>\delta>0$.
\begin{enumerate}
\item 
\begin{align}
F|_{W_\delta(\Sigma)\cap X}&=\bigoplus_i \tilde{L}^i|_{W_\delta(\Sigma)\cap X}\notag\\
&=\bigoplus_i (\beta|_{W_\delta(\Sigma)\cap X})^*\tilde{L}^i \notag\\
&= \beta|_{W_\delta(\Sigma)\cap X}^*(\iota^*E_0^+)\notag\\
&= (\epsilon_{\delta}|_{W_\delta(\Sigma)\cap X})^*E_0^+\notag
\end{align}
\item Recall that the curvature of $\tilde{L^i}$ is $\Omega^{\tilde{L^i}}=d\omega_i+\omega_i\wedge\omega_i=d\omega_i$. Thus
\[||\mathcal{R}^F||\leq\sum_i||\mathcal{R}^{\tilde{L^i}}||\leq 2^{n-1}C_0\max_i\{||\Omega^{\tilde{L}^i}||_\infty\}\]
Here $C_0$ is a constant which depends only on the dimension $n$. (See section 2 in \cite{gromov1983positive}.) Therefore
\begin{align}
||\mathcal{R}^F||&\leq 2^{n-1}C_0\max_i\{||d\omega_i||_\infty\}\notag\\
&\leq 2^{n-1}C_0 \max\{|log(\xi_i)|\cdot||d\omega||_\infty; \xi_i\in Hol(E_0^+|_\Sigma)\}\notag\\
&\leq 2^{n-1}C_0C_n\leq \frac{1}{4}n(n-1)\notag
\end{align}
(We can choose $C_n$ so that the last inequality holds.)
\item Using Proposition \ref{drham}
\begin{align}
\int_X \hat{A}(TX)\wedge\hat{ch}(F)&=\sum_i\int_X \hat{A}(TX)\wedge\hat{ch}(\tilde{L}^i)\notag\\
&=\sum_i\int_X ch_m(\tilde{L}^i)\notag\\
&=\sum_i\int_X \frac{1}{m!}c_1(\tilde{L}^i)^m=0\notag
\end{align}
Notice that $\mathrm{dim}(X)>4$ and $c_1(\tilde{L}^i)^{2}\in H_{cpt}^{4}(X)=0$ so $c_1(\tilde{L}^i)^{m}=c_1(\tilde{L}^i)^{2}\wedge c_1(\tilde{L}^i)^{m-2}=0\in H_{cpt}^{n}(X)$.
\end{enumerate}
This completes the proof.
\end{proof}
\subsection{Proof of The Main Theorem}
Finally, we investigate the sufficient condition for the hypothesis of Theorem \ref{ext}. We restate our main theorem.
\begin{thm}\label{main2}
There exists a universal constant $C_n$ such that the following statement holds.
Let $\Sigma$ be an embedded circle $\iota: S^1\hookrightarrow S^n$ and $X:=S^n\backslash\Sigma$.
Let $g$ be a complete Riemannian metric on $X$.
\[\inf_{x\in X}Sc(g)_x<n(n-1)\]
holds if the following three conditions hold.
\begin{enumerate}[label=(\arabic*)]
\item $n>4$ and $n\equiv 0 \mod 4$.
\item Any real Lipschitz chain $C$ which represents the unit element $[C]$ in $H_2(S^n, W(\Sigma); \mathbb{R})$ satisfies 
\[Area_g(C)>C_n\cdot \max\{|log(\xi_i)|; \xi_i\in Hol(E_0^+|_\Sigma)\}\]
\item $f:(X,g)\to (X,g_{std})$ is a 1-contracting diffeomorphism.
\end{enumerate}
\end{thm}
Before proving the main theorem, we begin by setting some notations.
\begin{itemize}
\item For $\phi\in\Omega^p(Y)$, \[||\phi||_\infty:=sup\{\frac{\phi(v_1,\dots,v_k)}{|v_1|\dots|v_k|};
p\in Y,v_1,\dots,v_k\in T_pY\backslash\{0\}\}\]
\item The space of $p$-dimensional currents with support in $Y$;
\[\mathbf{E}_p(Y):=\Omega^p(Y)^\ast\]
\item For $T\in\mathbf{E}_p(Y)$, \[M(T):=sup\{T(\phi)|||\phi||_\infty\leq 1\}\]
\item Set the normed space $\mathcal{E}_p(Y):=\{T\in\mathbf{E}_p(Y);M(T)<\infty\}$ with the mass norm $M(\cdot)$.
\item Set the subspace $S_p(Y)$ of $\mathcal{E}_p(Y)$ as \[S_p(Y):=\{T\in\mathbf{E}_p(Y); \text{T is a real Lipschitz chain}\}\]
\item Set $S_2(X_{\leq R+1}, X_{[R, R+1]}):=\{T\in S_2(X_{\leq R+1}); \partial T\in S_1(X_{[R, R+1]})\}$
\end{itemize}
Define $\delta:X\to\mathbb{R}$ as $X\cong S^{n-2}\times\mathbb{R}^2\to\mathbb{R}$ where the second map is 
$S^{n-2}\times\mathbb{R}^2\to\mathbb{R},(p,q)\mapsto |q|$. Define $\pi: X\to S^{n-2}$ as $X\cong S^{n-2}\times\mathbb{R}^2\to S^{n-2}$ where $\phi\circ f$ gives the diffeomorphism $X\cong S^{n-2}\times\mathbb{R}^2$ and $f:(X,g)\to (X,g_{std})$ is  the 1-contracting diffeomorphism given in the hypothesis of Theorem \ref{main2}. We use the following notations.\\
\[X_{[R_0,R]}=\delta^{-1}([R_0,R]),\ X_R=\delta^{-1}(R),\ X_{>R}=\delta^{-1}((R,\infty))\]
and \[X|_U=\pi^{-1}(U),\ X|_p=\pi^{-1}(p)\]
Set $\eta:=\beta^*d\theta$.
We recall the following Proposition which is essentially the same as Proposition 7.25. in \cite{gromov1983positive}.
\begin{prop}\label{ext2}
Let $\lambda$ be any positive real number.
  Suppose $\eta\in\Omega^1(X_{[R,R+1]})$ with the property that
  \[|(\partial T)(\eta)|\leq\lambda M(T)\]
  for all $T\in S_2(X_{\leq R+1}, X_{[R, R+1]})$. Then there exists an extension of $\eta$ to all of $Y$, that is, there exists $\tilde{\eta}\in\Omega^1(Y)$
  such that $i^\ast \tilde{\eta}=\eta$, with the property that
  \[||d\tilde{\eta}||_\infty\leq\lambda\]
\end{prop}
\begin{proof}
The proof is a slight modification of \cite{gromov1983positive} Proposition 7.25. We repeat the argument for the completeness of this paper.\\
Set $\mathcal{B}_0:=\partial S_2(X_{\leq R+1}, X_{[R, R+1]})$, $\mathcal{B}_1:=\partial S_2(X_{\leq R+1})$. 
Define a functional $F(\tau)=\int_\tau \eta$ on $\mathcal{B}_0$. Define the continuous semi-norm $p$ on $\mathcal{B}_1$ as
\[p(\tau)=\inf\{M(\sigma);\exists\sigma\in\mathcal{S}_2(X_{\leq R+1})\text{s.t. }\partial\sigma=\tau\}\]
By the Hahn-Banach theorem, there exists an extension $\tilde{F}:\mathcal{B}_1\to \mathbb{R}$ of $F$ such that $|\tilde{F}(\tau)|\leq\lambda\cdot p(\tau)$.\\
The second application of the Hahn-Banach theorem gives an extension $\tilde{F}$ on $\mathcal{E}_1(X_{\leq R+1})$. The corresponding 1-form $\tilde{\eta}$ of $\tilde{F}$ is what we want because for any small Lipschitz 2-simplex $\sigma$,
\[\int_\sigma (d\tilde{\eta}, vol)vol=\int_\sigma d\tilde{\eta}\leq \lambda M(\sigma)\]
Thus $||d\tilde{\eta}||_\infty\leq\lambda$.
\end{proof}
Consider for each $a$, the number
 \[e(a)=\inf\{M(T)|T\in S_2(X_{\leq R+1}, X_{[R, R+1]}), (\partial T)(\eta)=1\}\]
Set $e(\infty):=\sup \{\lim_{k\to\infty} e(R_k); R_1<R_2<\cdots \to \infty \} $(See (7.30) in \cite{gromov1983positive}.)

\begin{proof}[Proof of Theorem \ref{main2}]
We can assume $f$ is strictly contracting at infinity because we can replace $f$ by $f'=h\circ f$ where $h$ is a self-map on $\mathbb{S}^n$ which is contracting except for a neighborhood of some point. (Notice that there exists a point $p$ such that $dil(f)_p<1$.)\\
Set $C(n):=2C_n^{-1}$. 
We first claim $e(\infty)>C_n^{-1}\max_i\{|\theta_i|\}$. Suppose not, there exists $R_1<R_2<\cdots \to \infty$ such that 
$\lim_{k\to\infty} e(R_k)<2C_n^{-1}\max_i\{|\theta_i|\}$. 
Then there exists a real Lipschitz chain $C$ which represents the unit element $[C]$ in $H_2(S^n, W(\Sigma); \mathbb{R})$ satisfies 
\[Area_g(C)< C(n)\cdot \max_i\{|\theta_i|\}\] This contradicts to (2).\\
Thus there exists a sufficiently large $R>0$ 
such that \[e(R)\geq C_n^{-1}\cdot \mathrm{max}\{|log(\xi_i)|; \xi_i\in Hol(E_0^+|_\Sigma)\}\]By Proposition \ref{ext2}, there exists 
$\omega\in\Omega^1(X_{\leq R})$ such that $\omega|_{X_{[R, R+1]}}=\eta$ and 
\[\mathrm{max}\{|log(\xi_i)|; \xi_i\in Hol(E_0^+|_\Sigma)\}\cdot||d\omega||_\infty\leq C_n\]
By using Theorem \ref{ext},
\[\inf_{x\in X}Sc(g)_x<n(n-1)\]
\end{proof}

\section{Examples and Open Problems}
In this section, we examine some examples of Theorem \ref{main2}. The main theorem in this section states that Theorem\ref{main2} holds if any properly embedded surface $S$ bounding $\Sigma$ (i.e. $\partial S=\Sigma$) has $Area_g(S)=\infty$, instead of the assumption in Theorem\ref{main2}, under a suitable condition for metrics at infinity.\\\indent
First, we begin with some definitions and propositions. We continue to use notations of the previous section. For the rest of this section, let $\Sigma$ be an embedded circle $\iota: S^1\to S^n$ and $X:=S^n\backslash \Sigma$. We assume $n>4$ and $n\equiv 0 \mod 4$, and $f: (X,g)\to (X, g_{std})$ is a 1-contracting diffeomorphism covered by a diffeomorphism $\bar{f}: S^n\to S^n$ which fix $\Sigma$, $\bar{f}|_\Sigma=id_{\Sigma}$.
Throughout this section, we fix $R_0>0$ so that $f(X_{>R_0})\subset W_{1/100}(\Sigma)$.
\begin{dfn}\label{A}
Define $\mathcal{A}_R$ as the infimum of $L\geq 1$ over a smooth map $f: D\to E$ satisfying the following conditions where $D$ is a fundamental domain of $\tilde{X}_{[R_0,R]}$ is the universal cover of $X_{[R_0,R]}$.
\begin{itemize}
\item $E=\mathbb{R}^N\times\mathbb{D}^2(r)$ is the product of the Euclidean space and the standard disk with radius $r$, and $f$ is a bundle map covered by an embedding $S^{n-2}\to\mathbb{R}^N$.
\item $||\wedge^2df||_\infty\leq L$ and $||\wedge^2df^{-1}||_{f(D|_p)}\leq L$ for all $p\in S^{n-2}$.
\end{itemize}
Set $\mathcal{A}:=\limsup_{R\to\infty} \mathcal{A}_R$. 
\end{dfn}
Set $\mathcal{T}_R:=\{A|\exists p\in S^{n-2}s.t.\ A=X_{[R_0,R]}|_p\}$.
\begin{prop}
Suppose that 
\[\mathcal{A}_R^2<\frac{\lambda}{4\pi}\cdot \inf_{A\in\mathcal{T}_R}Area_g (A)\]
Then there exists $\tilde{\omega}\in\Omega^1(X_{[R_0,R]})$ such that $\tilde{\omega}|_{X_R}=\beta^*d\theta$, $\tilde{\omega}|_{X_{R_0}}=0$ and $||d\tilde{\omega}||_\infty<\lambda$. 
\end{prop}
\begin{proof}
Take $L>0$ such that $\mathcal{A}_R^2<L^2<\frac{\lambda}{4\pi}\cdot\inf_{A\in\mathcal{T}_R}Area_g(A)$. By the definition of $\mathcal{A}_R$, we can take a smooth map $f: D\to E$ satisfying the following conditions. (Let $D$ be a fundamental domain of $\tilde{X}_{[R_0,R]}$, i.e. $D\cong S^{n-2}\times [0,1]\times [R_0,R]$. )
\begin{itemize}
\item $E=\mathbb{R}^N\times\mathbb{R}^2$ is the euclidean space and $f$ is a bundle map covered by an embedding $S^{n-2}\to\mathbb{R}^N$.
\item $||\wedge^2df||_\infty\leq L$ and $||\wedge^2df^{-1}||_{f(D|_p)}\leq L$ for all $p\in S^{n-2}$.
\end{itemize}

We can construct $\bar{\beta}:\partial D\to S^1$ such that
\[\bar{\beta}|_{S^{n-2}\times[0,1]\times \{R\}}=\beta\circ (\tilde{X}_{[R_0,R]}\to X_{[R_0,R]})\]
and $\bar{\beta}=*$ on $S^{n-2}\times[0,1]\times \{R_0\}$ and $deg(\bar{\beta}|_{\partial D_p})=1$ for $p\in S^{n-2}$.\\
Notice that $Area(D_p)^{-1}<\frac{\lambda}{4\pi}\cdot L^{-2}$ for all $p\in S^{n-2}$. By considering the tubular neighborhood of $f$, our problem can be reduced to the following lemma.
\begin{lem}
Suppose that there exists a open $U\subset\mathbb{R}^N$ containing $f(X_0)$ and a disk bundle $D':=U\times \mathbb{D}(r)\subset E$ over $U$. (We think $E$ as a trivial bundle over $\mathbb{R}^N$.)\\
We have an inequality
\[Area(\mathbb{D}(r))^{-1}<\frac{\lambda}{4\pi}\cdot L^{-1}\]
for all $p\in U$. We have a map $\bar{\beta}':\partial D'\to S^1$ such that $deg(\bar{\beta}'|_{\partial D'_p})=1$ for $p\in U$. Then there exists an extension $\tilde{\omega}'\in\Omega^1(D')$ of $\omega=\bar{\beta}'^*d\theta$ such that
\[||d\tilde{\omega}'||_\infty<\lambda\cdot L^{-1}\]
\end{lem}
\begin{proof}
Without loss of generality, we can assume $\bar{\beta}'|_{\partial D'_p}:\partial D'_p\to S^1=\frac{1}{r}$ for all $p\in U$ because we can replace $\bar{\beta}'$ by using a homotopy near the neighborhood of $\partial D'$. Set a function $\rho:[0,r]\to [0,1]$  as $\rho\equiv 0$ on $[0,r/2]$ and $\rho(x)$ on $[r/2,r]$ is a smooth approximation of a linear function $\frac{2}{r}x-1$. We define $\bar{\rho}:D'\to [0,1]$ through
\[D'=U\times \mathbb{D}(r)\to \mathbb{D}(r)\to [0,r]\xrightarrow{\rho}[0,1]\].
The map $\mathbb{D}(r)\to [0,r]$ maps $p$ to the radial distance of $p$. We extend $\bar{\beta}'$ on $U\times (\mathbb{D}(r)\backslash\mathbb{D}(r/2))$ by using the retraction $U\times (\mathbb{D}(r)\backslash\mathbb{D}(r/2))$ to $\partial D'$. We define $\tilde{\omega}':=\bar{\rho}\bar{\beta}'^*(d\theta)$ and we have
\[||d\tilde{\omega}'||_\infty\leq ||d\bar{\rho}||\cdot Lip(\bar{\beta}')\cdot ||d\theta||\leq \frac{2}{r}\cdot\frac{2}{r}\leq \lambda\cdot L^{-1}\]
\end{proof}
Then $\tilde{\omega}:=f^*\omega'$ is what we want
\end{proof}

\begin{thm}\label{main3}
Suppose that $\mathcal{A}<\infty$.
Let $\Sigma$ be an embedded circle $\iota: S^1\hookrightarrow S^n$ and $X:=S^n\backslash\Sigma$.
Let $g$ be a complete Riemannian metric on $X$.
\[\inf_{x\in X}Sc(g)_x<n(n-1)\]
holds if the following three conditions hold.
\begin{enumerate}[label=(\arabic*)]
\item $n>4$ and $n\equiv 0 \mod 4$.
\item Any properly embedded surface $S$ bounding $\Sigma$ (i.e. $\partial S=\Sigma$) has $Area_g(S)=\infty$.
\item $f:(X,g)\to (X,g_{std})$ is a $1$-contracting diffeomorphism covered by a diffeomorphism $\bar{f}: S^n\to S^n$ which fix $\Sigma$, $\bar{f}|_\Sigma=id_{\Sigma}$.
\end{enumerate}
\end{thm}
\begin{proof}
First we claim $\lim_{R\to\infty} \inf_{A\in\mathcal{T}_R}Area_g(A)=\infty$. Suppose not, there exists $R_1<R_2<\cdots \to \infty$ such that $A_k=X_{[R_0,R_k]}|_{p_k}\in\mathcal{T}_{R_k}$ and 
\[Area_g (A_k)<\lim_{R\to\infty} \inf_{A\in\mathcal{T}_R}Area_g(A)<\infty\]
Then by the compactness of $S^{n-2}$, we can assume the convergence $p=\lim_k p_k$. Then, set $S$ as $S:=D\cup X_{\geq R_0}|_p$ where $D$ is any disk bounding $X_{\geq R_0}|_p$. This contradicts to (2) because
\[Area_g(S)=Area_g(D)+Area_g(X_{\geq R_0}|_p)\leq Area_g(D)+\lim_k Area_g(A_k)<\infty\]
Thus there exists a sufficiently large $R>0$ 
such that \[\inf_{A\in\mathcal{T}_R}Area_g(A)\geq \mathcal{A}\cdot C_n^{-1}\cdot \mathrm{max}\{|log(\xi_i)|; \xi_i\in Hol(E_0^+|_\Sigma)\}\]By Proposition \ref{ext2}, there exists 
$\omega\in\Omega^1(X_{\leq R})$ such that $\omega|_{X_{[R, R+1]}}=\eta$ and 
\[\mathrm{max}\{|log(\xi_i)|; \xi_i\in Hol(E_0^+|_\Sigma)\}\cdot||d\omega||_\infty\leq C_n\]
By using Theorem \ref{ext},
\[\inf_{x\in X}Sc(g)_x<n(n-1)\]
\end{proof}
First, we give the examples of Theorem \ref{main3}. Notice that our essential problem is the metric near the closed smooth curve. 
\begin{example}
Let $h$ be any complete metric on $W_{1/100}(\Sigma)\backslash{\Sigma}$ such that $h=g_{std}$ near the boundary $\partial W_{1/100}(\Sigma)$. Define a metic $g$ on $S^n\backslash{\Sigma}$ as the glued metric $g=g_{std}\cup h$.\\
If $h$ is periodic along the function $\delta$, i.e. the shift \[\sigma:(X_{[R_0,\infty)}, h)\to (X_{[R_0+1,\infty)},h)\] gives an isometry, then $\mathcal{A}<\infty$ and no properly embedded surface $S$ bounding $\Sigma$ has finite area with respect to $g$ because $g$ is of bounded geometry. The reason of $\mathcal{A}<\infty$ is as follows.\\\indent
First, because of the periodicity, there exists a constant $L_0$ which satisfies the followings. Fix any $R$, let $D$ be a fundamental domain of $\tilde{X}_{[R_0,R]}$ is the universal cover of $X_{[R_0,R]}$. There exists a  bundle map $F_1: D\to \mathbb{S}^{n-2}\times Q$ such that $||\wedge^2dF_1||_\infty\leq L_0$ and $||\wedge^2d{F_1}^{-1}||_{F_0(D_p)}\leq L_0$ for all $p\in S^{n-2}$ where $Q$ is a rectangle in the Euclidean plane with sizes $a$ and $b$.\\\indent
Secondly, define $r:=\max\{\sqrt{a/b},\sqrt{b/a}\}^{-1}$. Set a map $G_1:\mathbb{S}^{n-2}\to \mathbb{S}^{n-2}(r):=r\cdot id$ and an area-preserving map $G_2:Q\to R$ where $R$ is a square with sizes $\sqrt{ab}$. We can check the map $G:=G_1\times G_2$ satisfies $||\wedge^2 dG||_\infty\leq 1$.\\\indent
Thirdly, the map $F:=G\circ F_1$ has the properties in Definition \ref{A}. (Notice that the square with sizes $r$ is bi-Lipschitz to the disk $\mathbb{D}(r)$.)
\end{example}
The relationship between a family of surfaces in a given Riemannian manifold $X$ and the positive lower bound of the scalar curvature of $X$ is pointed out in the paper \cite{gromov2018dozen}. To state the precise statement of the conjecture in \cite{gromov2018dozen}, we recall the following definitions.
\begin{dfn}[Slicings and Waists \cite{gromov2018dozen}]
An \textbf{$m$-sliced $n$-cycle}, $m\leq n$, is an $n$-dimensional simplicial complex $P^n$ and a simplicial map $\varphi:P\to Q$  where $Q$ is an $(n-m)$-dimensional simplicial complex and where all preimages $P_q=\varphi^{-1}(q)\subset P$ have $\dim(P_q)\leq m$, $q\in Q$.\\
The \textbf{$m$-waist}, denoted $waist_m(h)$, of a homology class $h\in H_n(X;\mathbb{Z}/2)$ is the infimum of the numbers $w$, such that $X$ receives a Lipschitz map from a compact $m$-sliced $n$-cycle, $\varphi:P^n\to X$, which represent $h=\varphi_*[P]$ and $Vol_X(\varphi(P_q))\leq w$ for all $q\in Q$.
\end{dfn}
The following conjecture is one of the conjectures posed by Gromov in \cite{gromov2018dozen}.
\begin{conj}
Let $X$ be a closed $n$-dimensional Riemannian manifold. The scalar curvature of $X$ is bounded from below:
\[Sc(X)\geq Sc(\mathbb{S}^n)=n(n-1)\]
Then the slicing area of the fundamental homology class $[X]\in H_n(X;\mathbb{Z}/2)$ is
bounded by
\[waist_2[X]\leq const_{++}\]
(Ideally, it is expected that
$const_{++}=waist_2(S^n)$
where $waist_2[S^n]=area(S^2)=4\pi$.)
\end{conj}
Finally, we propose the following problem regarding the relationship between a family of surfaces in a given Riemannian manifold $X$ and the positive lower bound of the scalar curvature of $X$.
\begin{qest}
Is there a constant $c(K)$ so that the following statement holds ?\\\indent
Suppose that $K$ is a knot embedded in a sphere $S^3$. Let $g$ be a Riemannian metric on $S^3\backslash K$. Then the inequality 
\[\inf_{x\in S^3\backslash K} Sc(g)_x<6\]
holds if $g\geq g_{std}$ and 
\[Area_g(S)\geq c(K)\]
for all Seifert surfaces $S$ of $K$.
\end{qest}
For the context of quantitative topology, the volume growth in metric spaces is one of the important geometric invariants. (For example, the reader can consult \cite{guth2011volumes}.) There is another question analogous to Theorem \ref{main2}.
\begin{qest}
Suppose that $X$ is a compact $n$-dimensional CW-complex and Let $\Sigma$ be an embedded circle $\iota: S^1\hookrightarrow S^n$. Is there a constant $c(n, \Sigma)$ so that the following statement holds ?\\\indent
Let $f: (X,\partial X)\to (S^n, \Sigma)$ be a continuous map which gives an isomorphism $f_*: H_*(X, \partial X)\xrightarrow{\cong}H_*(S^n, \Sigma)$. Suppose that $f: Int(X)\to \mathbb{S}^n\backslash\Sigma$ is a proper 1-Lipschitz map and that any integral Lipschitz $2$-chain $C$ which is the unit element in $H_2(X,\partial X)$ satisfies
\[Area_X(C)\geq c(n,\Sigma)\]
Then $V_X(R)\geq V_{S^n}(R)$.
\end{qest}
\section*{Acknowledgements}
The author would like to thank his advisor Professor Tsuyoshi Kato for his constant encouragement and helpful suggestions. He would like to thank the professors Kaoru Ono and Nigel Higson for their valuable comments and helpful discussions.

\bibliographystyle{amsalpha}
\bibliography{reference2.bib}

\end{document}